\documentclass[reqno, 12pt, reqno]{amsart}

\usepackage{amsfonts, amsthm, amsmath, amssymb}
\usepackage{hyperref}
\hypersetup{colorlinks=false}

\usepackage[margin=3.4cm]{geometry}
\usepackage{stmaryrd}
\usepackage{enumitem}
\usepackage{tikz-cd}
\usetikzlibrary{cd}

\RequirePackage{mathrsfs}
\let\mathcal\mathscr

\numberwithin{equation}{section}

\newtheorem{theorem}{Theorem}[section]
\newtheorem{lemma}[theorem]{Lemma}
\newtheorem{proposition}[theorem]{Proposition}

\theoremstyle{remark}
\newtheorem*{ack}{Acknowledgements}
\newtheorem{remark}[theorem]{Remark}

\newtheorem{definition}[theorem]{Definition}

\renewcommand{\phi}{\varphi}
\renewcommand{\rho}{\varrho}
\renewcommand{\epsilon}{\varepsilon}

\renewcommand{\leq}{\leqslant}
\renewcommand{\le}{\leqslant}
\renewcommand{\geq}{\geqslant}
\renewcommand{\ge}{\geqslant}

\newcommand{\PP}{\mathbb{P}}

\newcommand{\ZZ}{\mathbb{Z}}
\newcommand{\NN}{\mathbb{N}}
\newcommand{\QQ}{\mathbb{Q}}
\newcommand{\RR}{\mathbb{R}}

\newcommand{\cG}{\mathcal{G}}
\newcommand{\cS}{\mathcal{S}}
\newcommand{\cQ}{\mathcal{Q}}

\newcommand{\cP}{\mathcal{P}}
\newcommand{\cO}{\mathcal{O}}

\newcommand{\cL}{\mathcal{L}}

\newcommand{\cT}{\mathcal{T}}

\newcommand{\cHom}{\mathscr{H}\!om}

\newcommand{\cUL}{\mathcal{U\!L}}
\newcommand{\cUP}{\mathcal{U\!P}}
\newcommand{\cUG}{\mathcal{UG}}

\newcommand{\x}{\mathbf{x}}
\newcommand{\y}{\mathbf{y}}
\newcommand{\uu}{\mathbf{u}}
\newcommand{\vv}{\mathbf{v}}
\newcommand{\z}{\mathbf{z}}
\newcommand{\w}{\mathbf{w}}
\newcommand{\bb}{\mathbf{b}}

\newcommand{\al}{\alpha}
\newcommand{\ga}{\gamma}
\newcommand{\dl}{\delta}
\newcommand{\lm}{\lambda}
\newcommand{\ve}{\varepsilon}

\DeclareMathOperator{\rank}{rank}
\DeclareMathOperator{\Pic}{Pic}
\DeclareMathOperator{\Spec}{Spec}

\DeclareMathOperator{\covol}{covol}

\DeclareMathOperator{\Span}{span}
\DeclareMathOperator{\vol}{vol}
\DeclareMathOperator{\Hilb}{Hilb}
\DeclareMathOperator{\Hom}{Hom}
\DeclareMathOperator{\SL}{SL}
\DeclareMathOperator{\GL}{GL}
\DeclareMathOperator{\PGL}{PGL}
\DeclareMathOperator{\Mat}{Mat}
\DeclareMathOperator{\Gr}{Gr}
\newcommand{\gr}{\Gr}
\DeclareMathOperator{\id}{id}

\newcommand{\idmat}[1]{\operatorname{I}_{#1}}
\newcommand{\diag}[1]{\operatorname{diag} (#1)}
\newcommand{\indMat}[1]{\operatorname{Mat}_{#1}^{\times}}
\newcommand{\induniMat}[1]{\operatorname{Mat}_{#1}^{\times,1}}

\newcommand{\transpose}{\mathrm{t}}
\newcommand{\dual}[1]{{#1^{*}}}
\newcommand{\perpen}[1]{{#1^{\perp}}}
\newcommand{\factor}[1]{#1^{\pi}}
\newcommand{\factorg}[2]{{#1^{\pi,#2}}}
\newcommand{\uni}[1]{#1^{1}}

\newcommand{\norm}[1]{\left\lVert#1\right\rVert_2}

\newcommand*\diff{\mathop{}\!\mathrm{d}}

\begin{document}

\title{Equidistribution and freeness 
 on Grassmannians}

\author{Tim Browning}
\author{Tal Horesh}
\author{Florian Wilsch}

\address{IST Austria\\
Am Campus 1\\
3400 Klosterneuburg\\
Austria}
\email{tdb@ist.ac.at}
\email{tal.horesh@ist.ac.at}
\email{florian.wilsch@ist.ac.at}


\begin{abstract}
We associate a certain tensor product lattice to any primitive integer lattice and ask about its typical 
shape. These lattices are related to the tangent bundle of  Grassmannians and their study is motivated 
by  Peyre's  programme on ``freeness'' for rational points of bounded height on Fano varieties. 
\end{abstract} 

\date{\today}

\maketitle

\thispagestyle{empty}
\setcounter{tocdepth}{1}
\tableofcontents

\section{Introduction}

Understanding the density of rational points on smooth Fano varieties lies at the confluence of algebraic geometry, harmonic analysis and analytic number theory. The guiding conjecture is due to Manin~\cite{FMT89}, and its refinement by  Peyre~\cite{peyre-duke}. 
Given a smooth Fano variety $V$ defined over a number field $k$ such that 
$V(k)$ is Zariski dense in $V$, and given an anticanonical height function 
$H\colon V(k)\to \RR$, 
it is conjectured that there exists a thin set of rational points $Z\subset V(k)$ 
and an explicit
 constant $c_{V,H}>0$ 
such that 
\begin{equation}\label{eq:manin}
\#\{x\in V(k)\setminus Z: H(x)\leq B\} \sim c_{V,H} B(\log B)^{\rho(V)-1}, \quad \text{as } B \to \infty,
\end{equation}
where $\rho(V)=\rank \Pic(V)$.
(Here, as proposed by Serre~\cite[\S~3.1]{Ser08}, a subset $Z \subset V(k)$ is said to be \emph{thin} if it is
a finite union of subsets which are either contained in a proper closed subvariety of $V$, or contained 
in some $\pi(Y(k))$ where $\pi\colon Y \to V$ is a generically finite dominant morphism, with $\deg(\pi)>1$ and  $Y$ irreducible.)

Despite major progress by 
Lehmann, Sengupta and Tanimoto~\cite{geometry} on identifying problematic thin sets for Fano varieties, they can be hard to work with in general.  An alternative path based on a notion  of ``freeness'' has surfaced in recent work of 
Peyre~\cite{peyre-freedom}. It posits the idea that the distribution of ``sufficiently free'' rational points of bounded height on smooth Fano varieties should conform to the
asymptotic behaviour in~\eqref{eq:manin}, without the need to first identify appropriate thin sets
of rational points.  This  conjecture has been confirmed for smooth hypersurfaces over $\QQ$ of low degree by Browning and Sawin~\cite{sawin3}. 
However, at the same time it has also been shown by Sawin~\cite{sawin} that 
the proposal fails for the Hilbert scheme $\Hilb^2(\PP^n)$,  since the thin subset  of rational points consisting of points that lift to a certain double cover contains many points with relatively large freeness.
 (In fact, as discussed in \cite[\S~4]{peyre-a}, Peyre has supplemented his freeness proposal with an ``all the heights'' variant, which explains this example.)

The primary  aim of this paper is to study  the notion of freeness in the classical setting  of
Grassmannian varieties over $\QQ$. Not only  shall we be able to show that Peyre's freeness conjecture holds true for Grassmannians, but we shall even be able 
to prove a natural equidistribution statement for certain tangent lattices that emerge in the definition  of freeness. 

\medskip

For any integers $1\leq m\leq n-1$, 
there is a natural $\QQ$-scheme $\gr(m,n)$ whose $\QQ$-rational points $\gr(m,n)(\QQ)$ coincide with $m$-dimensional linear subspaces of $\QQ^n$. 
The Grassmannian is a smooth $\QQ$-variety of dimension $m(n-m)$, admitting a  
Pl\"ucker embedding 
\[
    \iota\colon \gr(m,n)\to \PP_\QQ^{\binom{m}{n}-1}.
\]  
The Picard group $\Pic (\gr(m,n))$ is isomorphic to $ \ZZ$ and it is generated by the divisor 
$\cL=\iota^* \cO(1)$. The anticanonical bundle is $\omega^{-1}=\cL^{\otimes n}$. This  confirms that $\gr(m,n)$ is Fano and allows one to define an
anticanonical height function  $H=H_{\omega^{-1}}$ on $\gr(m,n)(\QQ)$.
On setting  $\Lambda = V \cap \ZZ^n$ for a linear subspace $V$, 
a rational point $x\in \gr(m,n)(\QQ)$ is the same thing as a primitive  lattice $\Lambda\subset \ZZ^n$ of rank $m$. Then $H(x)=\covol(\Lambda)^n$, where  
$\covol(\Lambda)$ is the covolume of the lattice.
Schmidt~\cite{schmidt} has shown that
\begin{equation}\label{eq:schmidt}
\# \{ x\in \gr(m,n)(\QQ): H(x)\leq B \}
= c_{m,n} B +O\left(B^{ 1 - \max\left\{\frac{1}{mn}, \frac{1}{(n-m)n}\right\}} \right),
\end{equation}
 for any $1\leq m\leq n-1$, 
where 
\begin{equation}\label{eq:cmn}
c_{m,n}=\frac{1}{n}\binom{n}{m}\frac{V(n)V(n-1)\dots V(n-m+1)}{V(1)V(2)\dots V(m)}
\frac{\zeta(2)\dots \zeta(m)}{\zeta(n)\dots \zeta(n-m +1)}
\end{equation}
and 
$V(N)=\vol\{\x\in \RR^N: \norm{\x}\leq 1\}$
for any $N\in \NN$.
This agrees with the conjectured asymptotic formula~\eqref{eq:manin}, with $Z=\emptyset$.
Moreover, Peyre~\cite[\S~6]{peyre-duke} has shown 
that the  leading constant $c_{m,n}$ agrees with his prediction.
In fact, Franke, Manin and Tschinkel~\cite{FMT89} have a far-reaching generalisation of Schmidt's result to arbitrary Flag varieties over arbitrary number fields and Peyre has confirmed his conjectured constant for this more general class.

In~\cite[D\'ef.~4.11]{peyre-freedom},
Peyre defines a freeness function $\ell(x)$ on  points $x\in V(\QQ)$,
for any smooth Fano $\QQ$-variety. 
This function takes values in $[0,1]$ and its precise definition will be  recalled in~\eqref{eq:mu-free}.
When $V=\gr(m,n)$ is the Grassmannian it admits a particularly succinct description. 
Associated to $x\in \gr(m,n)(\QQ)$ is the tensor product lattice
\[
    T_\Lambda= \Lambda^*\otimes_
    \ZZ \Lambda^\pi ,
\]
where $\Lambda\subset \ZZ^n$ is the primitive rank $m$ lattice  associated to $x$, 
with $\Lambda^*$ being the \emph{dual lattice} and 
$\Lambda^\pi=(\Lambda^\perp)^*$  the \emph{factor lattice}. (The precise definitions of these will be recalled in Section~\ref{s:gn}.) The lattice $T_\Lambda$
has rank $m(n-m)$ and will henceforth be referred to as the \emph{tangent lattice}.
The freeness function $\ell(x)$ measures the extent to which $T_\Lambda$ is lopsided, with a small value of $\ell(x)$ corresponding to the existence of an unusually large largest successive minimum.

Suppose for the moment that $m=1$, so that $\gr(1,n)=\PP_\QQ^{n-1}$. 
In this case it  follows from Peyre~\cite[Prop.~7.1]{peyre-freedom} that 
\begin{equation}\label{eq:Pn-1}
    \ell(x)\geq \frac{n-1}{n}
\end{equation}
for rational points  $x\in \PP_\QQ^{n-1}(\QQ)$.   
The  same bound  applies to the case $m=n-1$, by duality.
For the remaining values of $m$, it turns out that the freeness function can become arbitrarily small as one goes over rational points on the Grassmannian $\gr(m,n)$. In fact we have the following result. 
\begin{theorem}\label{t:unfree}
    Let $m,n$ be integers such that 
    $1< m<n-1$. Then there exist infinitely many $x\in \gr(m,n)(\QQ)$ such that $\ell(x)=0$.
\end{theorem}

Let $1\leq k<m$. 
Suppose we fix a $k$-dimensional linear subspace $L$ of $\QQ^n$.
Then among the linear subspaces parameterised by  $\gr(m,n)(\QQ)$ is the set $V_L$ of those that contain $L$.
Any point in $V_L$ whose height is large compared to that of $L$ will have low freeness. The example  constructed to prove Theorem~\ref{t:unfree} is of exactly this sort.
Although all points of low freeness have this kind of special structure, they do not form a thin subset.
\begin{theorem}\label{t:unfree-not-thin}
    Let $m,n$ be integers such that $1<m<n-1$.
    Then, for all $\ve>0$, the set
    \[
        \Omega_\ve = \{x\in \gr(m,n)(\QQ) : \ell(x)<\ve\}  
    \]
    of non-$\ve$-free points is not thin.
\end{theorem}

For a suitable range of $\ve>0$, Peyre suggests focusing on  the restricted counting function which only counts
\emph{$\ve$-free points}, meaning those points $x\in V(\QQ)$ for which  $\ell(x)\geq \ve$.
For suitable Fano varieties it is expected that such a restriction captures the behaviour articulated in the Manin--Peyre prediction. The following result confirms this for Grassmannians.

\begin{theorem}\label{t:free}
For any integers
$1\leq m\leq n-1$ and any $0\leq \ve<1$, we have 
\[
    \#\{x\in \gr(m,n)(\QQ):H(x)\leq B, ~\ell(x)\geq \ve\}=c_{m,n} B +O\left(B^{1-\frac{1-\ve}{n}}  (\log B)^{2m-2}\right),
\]
as $B\to \infty$, 
where $c_{m,n}$ is given by~\eqref{eq:cmn}.
\end{theorem}

We note that Theorem~\ref{t:free} is a direct consequence of~\eqref{eq:schmidt} when
$m=1$ or $m=n-1$, at least if $\ve\leq \frac{n}{n-1}$. We shall give two distinct proofs of the 
asymptotic formula
$\#\{x\in \gr(m,n)(\QQ):H(x)\leq B, ~\ell(x)\geq \ve\}\sim c_{m,n} B$. The first is based on adapting the geometry of numbers arguments developed by Schmidt~\cite{schmidt} and leads to the power saving   error term present in Theorem~\ref{t:free}. 
This is the object of 
Section~\ref{s:free-pf}. 
The second proof occurs in Section~\ref{ssec:equi-and-freeness}, but it doesn't give a power saving.
We have decided to give the second proof since it  is  based on a general  equidistribution statement of independent interest, 
which is  discussed in 
Section~\ref{s:equidistribution} and which gives access to further results. 
Indeed, using  this equidistribution statement, we can also address a question raised by Peyre~\cite[\S~9]{peyre-freedom} concerning an alternative counting function, in which one orders rational points by the \emph{maximal slope} $\mu_{\max}(T_\Lambda)$ of their tangent lattice, as defined 
in Section~\ref{s:gn}, instead of by height. In the same spirit as counting sufficiently free points, this approach excludes most points whose tangent lattice is very lopsided, since  such points
lead to tangent lattices whose maximal slope is large compared to the height. We shall prove the following result in Section~\ref{ssec:equi-and-freeness}.
\begin{theorem}\label{t:max-slope-intro}
For any integers
$1\leq m\leq n-1$, we have 
    \[
        \# \{\Lambda \in \Gr(m,n)(\QQ) : \mu_{\max}(T_\Lambda) \le \log B\} = c_{m,n}^\prime B^{m(n-m)} + o\left(B^{m(n-m)}\right),
    \]
    where $0<c_{m,n}^\prime < c_{m,n}$ is given  in~\eqref{eq:c'_mn}.
\end{theorem}
In particular, this confirms the expectation expressed in~\cite[Rem.~7.8]{peyre-freedom} for projective space.
This alternative counting function behaves quite differently from the standard counting function with respect to an anticanonical height function. In particular, it is compatible with products by ~\cite[Rem.~7.21(b)]{peyre-freedom}. Thus it follows from 
Theorem~\ref{t:max-slope-intro} that
\[
    \# \{x \in X(\QQ) : \mu_{\max}(T_x) \le \log B\} = c_{X} B^{\dim X} + o\left(B^{\dim X}\right),
\]
for a suitable constant $c_X>0$, when $X$ is a product  of Grassmannians. A striking feature of this example is that the rank of $\Pic(X)$, which appears in~\eqref{eq:manin}, can become arbitrarily large.

\medskip

Returning to his asymptotic formula~\eqref{eq:schmidt}, Schmidt~\cite{schmidt1,schmidt2} also  proved various counting statements about lattices drawn from the set
\begin{equation}\label{eq:Lmn}
    \cL_{m,n} =
    \{\text{$\Lambda\subset \RR^n$ is a lattice of rank $m$}\}.
\end{equation}
This  is a homogeneous space for the group $\SL_n(\RR)$
and carries an $\SL_n(\RR)$-invariant measure that is unique up to multiplication by a positive scalar and induced by a Riemannian metric.
The space $\cL_{m,n}$ has infinite volume with respect to this measure, but rescaling induces a natural projection $\cL_{m,n}\to \cUL_{m,n}$ to the subspace
\[
    \cUL_{m,n}= \{\Lambda\in \cL_{m,n}: \covol(\Lambda)=1\}
\]
of \emph{unimodular lattices}.
Restricting the Riemannian metric on $\cL_{m,n}$ to $\cUL_{m,n}$ endows the latter with a finite measure, which we normalise to a probability measure.
Schmidt focuses on proving  
equidistribution statements about the projections of the lattices in $\cL_{m,n}$ to the space of 
similarity classes of rank $m$ lattices. However, as 
worked out by 
Horesh and 
Karasik~\cite[Thm.~3.2~(3)]{HK1}, 
it is also possible to show that primitive lattices in 
$\cL_{m,n}$  equidistribute in 
$\cUL_{m,n}$,
as $\covol(\Lambda)\to \infty$.
Theorem~\ref{t:free} tells us that among the rational points $x\in \gr(m,n)(\QQ)$,
it is rare to find points whose associated tangent lattice $T_\Lambda=\Lambda^*\otimes_\ZZ \Lambda^\pi$ is very lopsided. 
Our final result is an equidistribution statement about  the set of tangent lattices
$\{T_\Lambda : \Lambda\in \Gr(m,n)\}$,
as  $H(\Lambda)\to \infty$.
Such a result fits into the landscape of recent work by Aka, Einsiedler and Shapira \cite{AES1,AES2},  which is concerned with the equidistribution of the lattices $\mathbf{v}^\perp$, as one runs over  primitive integer solutions $\mathbf{v}\in \ZZ_{\text{prim}}^n$ to the equation $v_1^2+\cdots+v_n^2=N$, as $N$ runs to  infinity over suitable sets of positive integers. Indeed, although we shall not do so here, the lattices 
$\mathbf{v}^\perp$ can be interpreted as tangent lattices associated to points on the affine quadric. 

For $\Lambda \in \gr(m,n)$,  the tangent lattices $T_\Lambda$ are
lattices of rank $m(n-m)$ in $\RR^n\otimes \RR^n$. They 
belong to the 
submanifold
\begin{equation*}
    \cG_{m,n}
    = \{\Lambda \otimes \Lambda^\perp : \Lambda \in \cL_{m,n} \}
    \subset \cL_{m(n-m),n^2},
\end{equation*}
a precise description of which is given in Section~\ref{s:equidistribution}.
We shall construct a natural 
probability measure on 
    $\cUG_{m,n} = \cG_{m,n} \cap \cUL_{m(n-m),n^2}$,
which we denote by $\nu$, paving the way to a proof of the following equidistribution result. 

\begin{theorem}\label{t:equi}
With respect to $\nu$,
the projections of the set
\[
  \{T_\Lambda : \Lambda \in \Gr(m,n)(\QQ)\}  
\]
to $\cUL_{m(n-m),n^2}$ 
equidistribute in the manifold $\cUG_{m,n}$, as $H(\Lambda)\to \infty$, with rate of convergence  
$O (H(\Lambda)^{-\frac{1}{16n^4}})$. 
\end{theorem}

The explicit  equidistribution statement is given below in 
Theorem~\ref{t:equi'}, with a variant phrased in terms of probability measures in Theorem~\ref{t:equi''}. The proof takes place 
in Section~\ref{s:equidistribution} and 
involves a reduction to an equidistribution statement about
pairs of lattices $(\Lambda,\Lambda^\pi)$,
for primitive  $\Lambda \in \cL_{m,n}$, 
that appears in recent work of
Horesh and Karasik~\cite[Thm.~3.2~(4)]{HK1}.

\begin{ack}
The authors are very grateful to Pierre Le Boudec and Will Sawin for useful comments, 
as well as to the anonymous referee for several helpful remarks. 
While working on this paper the first two authors were
supported by EPSRC 
grant \texttt{EP/P026710/1}, and the first and last authors by 
FWF grant P~32428-N35.  
\end{ack}

\section{Basic facts about lattices}\label{s:gn}

Recall that a lattice $\Lambda$ is a discrete subgroup of $\RR^N$, or more generally, of an $N$-dimensional real vector space equipped with an inner product. 
There exists an integer $r\leq N$ and linearly independent vectors 
$\mathbf{b}_1,\dots,\mathbf{b}_r\in \RR^N$ such that 
$\Lambda=\Span_\ZZ\{\mathbf{b}_1,\dots \mathbf{b}_r\}$. We then say that 
the \emph{rank} of $\Lambda$ is  $\rank(\Lambda)=r$ and call $\Lambda$ a \emph{full} lattice if $r=N$. 
The \emph{covolume} is $\covol(\Lambda)=\sqrt{\det(B^{\transpose} B)}=\sqrt{\det(\langle \bb_i, \bb_j \rangle)_{i,j}}$, 
where the \emph{basis matrix}
$B$ is the $N \times r$ matrix formed from the column vectors $\mathbf{b}_1,\dots,\mathbf{b}_r$. 
This definition is independent of the choice of basis.
We say that a lattice $\Lambda$ is \emph{unimodular} if $\covol(\Lambda)=1$. It will be convenient to denote 
by $\Lambda_\RR$  
the subspace $\Span_\RR(\Lambda)$  
 that a lattice $\Lambda$ generates.

A sublattice $\Lambda$ of a lattice $\Gamma$ is said to be \emph{primitive in $\Gamma$}
if there is no other sublattice $\Lambda'\subset \Gamma$ of the same rank which 
properly contains $\Lambda$. 
A lattice $\Lambda\subset\RR^N$ is said to be \emph{integral} if it is contained
in $\ZZ^N$. 
We say that an integral lattice is primitive if it is primitive in $\ZZ^N$.
Note that an integral lattice 
$\Lambda$ is primitive if and only if 
$\Lambda_\RR \cap \ZZ^N = \Lambda$.

In this section we collect together some facts about lattices, 
most of which are taken from the book by Cassels~\cite{cassels}. 
A special role in our work will be played by the dual lattice and the factor 
lattice, and so we begin by defining these.

\subsection*{Dual, orthogonal, and factor lattice}
Let $\Lambda\subset \RR^N$ be a lattice of rank $r$ with basis matrix $B$. 
The \emph{dual lattice} is defined to be 
\[
    \Lambda^*=
    \{\x\in \Lambda_\RR 
        : \langle \x, \y\rangle \in \ZZ \text{ for all $\y\in \Lambda$}\}.
\]
This lattice has basis matrix $B(B^{\transpose}B)^{-1}$.
It immediately follows that $\rank(\Lambda^*)=r$ and
$\covol(\Lambda^*)=\covol(\Lambda)^{-1}$.
Moreover, we have $(\Lambda^{*})^*=\Lambda$.

Now suppose that  $\Lambda\subset \ZZ^N$ is a primitive lattice.
The \emph{orthogonal lattice} $\Lambda^{\perp}$ is the primitive lattice
\begin{equation*}
    \Lambda^{\perp}
    = \left\{ \mathbf{a} \in \mathbb{Z}^N
        :  \langle \mathbf{a}, \mathbf{z} \rangle = 0
        \text{ for all $\mathbf{z} \in \Lambda$}
    \right\}.
\end{equation*}
We have $(\Lambda^\perp)^\perp=\Lambda$ and 
$\covol(\Lambda)=\covol(\Lambda^\perp)$.

If $\pi\colon \RR^N \to \Lambda_\RR^\perp$ is the orthogonal 
projection, then we define the \emph{factor lattice}
$\Lambda^\pi$ to be the projection $\pi(\ZZ^N)$.
We have $\Lambda^\pi=(\Lambda^\perp)^*$ and so
\[
    \covol(\Lambda^\pi)= \frac{1}{\covol(\Lambda)}.
\]

\subsection*{Successive minima and slopes}

Let  $\Lambda \subset \RR^N$ be a lattice of rank $r$.
For each $1\leq k\leq r$, let $s_k(\Lambda)$ be the least $\sigma>0$ such that 
$\Lambda$ contains at least $k$ linearly independent vectors of Euclidean length
bounded by $\sigma$. 
The $s_k(\Lambda)$ are the \emph{successive minima} of $\Lambda$, and they 
satisfy 
\[
    0<s_1(\Lambda)\leq s_2(\Lambda)\leq \cdots \leq s_r(\Lambda).
\]
(Note that when we speak of successive minima, we shall always mean with respect to the Euclidean norm $\norm{\cdot}$.)
It follows from Minkowski's second convex body theorem~\cite[\S~VIII.1]{cassels}
that
\begin{equation}\label{eq:upper-lower}
    \covol(\Lambda) \leq \prod_{i=1}^r s_i (\Lambda)\ll_N \covol(\Lambda), 
\end{equation}
where the implied constant depends only on the dimension $N$ of the ambient vector space.
Appealing to work of 
Banaszczyk~\cite[Thm.~2.1]{ban}, we have 
\begin{equation}\label{eq:succ-minima-duality}
1\leq s_k(\Lambda) s_{r-k+1}(\Lambda^*) \leq r ,
\end{equation}
for $1\leq k\leq r$.

The \emph{slope} of a lattice $\Lambda\subset \RR^N$ of rank $r$ is defined to be
\[
\mu(\Lambda)=-\frac{1}{r}\log \covol(\Lambda).
\]
The 
\emph{maximal slope} $\mu_{\text{max}}(\Lambda)$  of $\Lambda$ is the 
maximum of the slopes of all non-zero sublattices of $\Lambda$. On the other hand, 
the \emph{minimal slope} $\mu_{\text{min}}(\Lambda)$  of $\Lambda$ is the 
minimum of the slopes of all quotients $\Lambda/\mathrm{M}$, as $\mathrm{M}\subset \Lambda$ runs over proper and primitive sublattices.
According to 
Bost and Chen~\cite[Eq.~(2)]{bost-chen},  the maximal and minimal slopes are related via the formula
\begin{equation}\label{eq:slope-duality}
\mu_{\text{max}}(\Lambda^*) = - \mu_{\text{min}}(\Lambda),
\end{equation}
where $\Lambda^*$ is the dual of $\Lambda$. 
Appealing to Theorems 1 and 3  of Borek~\cite{borek}, we can also deduce that 
\begin{equation}\label{eq:borek}
    \begin{aligned}
        0 & \leq \log s_{r}(\Lambda)+ \mu_{\text{min}}(\Lambda)\leq c_r \quad \text{and}\\
        0 & \leq \log s_{1}(\Lambda)+ \mu_{\text{max}}(\Lambda)\leq c_r,
    \end{aligned}
\end{equation}
for a certain explicit constant $c_r>0$ depending only on the rank of $\Lambda$.

\subsection*{Tensor products}

Suppose we are given two lattices $\Lambda_1$ and $\Lambda_2$ in $\RR^N$, with $\rank(\Lambda_i)=r_i$  for $i=1,2$.
We may form the tensor product lattice 
$\Lambda_1\otimes_\ZZ \Lambda_2$, which has rank $r_1r_2$. 
Recall that $\Lambda_1 =(\Lambda_1^{*})^*$. 
Then we see that 
$\Lambda_1\otimes_\ZZ \Lambda_2=
(\Lambda_1^{*})^*\otimes_\ZZ \Lambda_2$, which  is isomorphic to the space 
$\Hom_\ZZ(\Lambda_1^*,\Lambda_2)$ via the map which takes an elementary tensor 
$\phi\otimes v_2\in (\Lambda_1^{*})^*\otimes_\ZZ \Lambda_2$
to the linear map $\Lambda_1^*\to \Lambda_2$ defined by 
$(\phi\otimes v_2)(v_1)=\phi(v_1)v_2$, 
for $v_1\in \Lambda_1^*$ and $v_2\in \Lambda_2$.
The covolume of the tensor product lattice is 
\begin{equation}\label{eq:tensor-covolume}
\covol(\Lambda_1\otimes_\ZZ \Lambda_2)
= \covol(\Lambda_1)^{r_2}\covol(\Lambda_2)^{r_1}.
\end{equation}
Several authors have investigated the relationship between the maximal slopes of tensor product lattices and the maximal slopes of the factors.
Appealing to  Chen~\cite[Thm.~1.1]{chen}, for example, one finds that 
\[
    \mu_{\max}(\Lambda_1)+ \mu_{\max}(\Lambda_2)\leq 
    \mu_{\max}(\Lambda_1 \otimes_\ZZ \Lambda_2)\leq
    \mu_{\max}(\Lambda_1)+ \mu_{\max}(\Lambda_2)+r_1+r_2.
\]
Applying this to $(\Lambda_1 \otimes_\ZZ \Lambda_2)^*=
\Lambda_1^*\otimes_\ZZ \Lambda_2^*$, we deduce that
\[
    \mu_{\min}(\Lambda_1)+ \mu_{\min}(\Lambda_2)-r_1-r_2\leq 
    \mu_{\min}(\Lambda_1 \otimes_\ZZ \Lambda_2)\leq
    \mu_{\min}(\Lambda_1)+ \mu_{\min}(\Lambda_2).
\]
Once taken in conjunction with~\eqref{eq:borek}, this implies that 
\begin{equation}\label{eq:s-tensor}
    s_{r_1}(\Lambda_1)s_{r_2}(\Lambda_2)\ll 
    s_{r_1r_2}(\Lambda \otimes_\ZZ \Lambda_2)\ll
    s_{r_1}(\Lambda_1)s_{r_2}(\Lambda_2),
\end{equation}
where the implied constants depend only on $r_1$ and $r_2$.

\subsection*{Freeness}
Suppose for the moment that $V$ is a smooth Fano $\QQ$-variety  
of dimension $r$, extended to a smooth 
projective 
scheme $V_\ZZ$ over $\Spec(\ZZ)$.
Any rational point $x \in V(\mathbb Q)$ extends to a unique integral point
$x \in V_\ZZ(\mathbb Z)$  of this scheme.
The pullback $(\cT_{V_\ZZ})_x$ of its tangent bundle $\cT_{V_\ZZ}$
along ${x}$ is a 
free $\ZZ$-module of rank $r$ inside
$(\cT_V)_x \otimes \RR = (\cT_{V_\ZZ})_x\otimes \RR$.
Fixing a Riemannian metric on $V(\mathbb R)$ induces an inner product on
$(\cT_V)_x\otimes \RR$ and makes $(\cT_{V_\ZZ})_x$ a lattice.
The model $V_\ZZ$ induces norms on $(\cT_V)\otimes K_v$ at all finite places 
$K_v$, and we get an \emph{adelic metric} on
$\cT_V$, as in~\cite[Ex.~3.4]{peyre-freedom}. 
The adelic metric on the tangent bundle induces an adelic norm on the
anticanonical bundle $\omega_V^\vee = \bigwedge^r \cT_V$,
hence an anticanonical height.
The logarithmic  anticanonical height is defined
in~\cite[Déf.~3.11]{peyre-freedom}, and it satisfies
\begin{equation*}
    h(x) = - \log\covol((\cT_V)_x) = r\mu((\cT_V)_x),
\end{equation*}
by~\cite[Déf.~4.1, Rem.~4.2]{peyre-freedom}.
In~\cite[D\'ef.~4.11]{peyre-freedom},
Peyre defines the \emph{freeness}  of $x$ to be  
\begin{equation}\label{eq:mu-free}
    \ell(x) =
    \frac{
        \max\left\{\mu_{\min} ((\cT_V)_{x}), 0\right\}
    }{
        \mu((\cT_V)_{x})
    } =
    \frac{
        \max\left\{r \mu_{\min} ((\cT_V)_{x}), 0\right\}
    }{
        h(x)
    }.
\end{equation}
Since the minimal slope is bounded by the slope, we have
\[
    0\leq \ell(x)\leq  1.
\]

\section{Tangent lattices for  Grassmannians}

We now interpret the above  in the case  $V=\gr(m,n)$ of  Grassmannians.
A rational point $x\in V(\QQ)$ is the same thing as a primitive 
lattice $\Lambda\subset \ZZ^n$ rank $m$. 

Following~\cite[\S~3.2.3--3.2.4]{3264}, for example, we may now indicate the construction of  
the tangent bundle of $\Gr(m,n)$. 
Consider the
trivial bundle $\cO_{V}^{\oplus n}$, carrying the standard inner product.
It admits the subbundle $\cS \subset \cO_{V}^{\oplus n}$ whose fibre at a point
$\Lambda$ is $\Lambda_\RR\subset \RR^n$, and the quotient bundle 
$\cQ=\cO_{V}^{\oplus n}/\cS$.
Then 
\[
    \cT_{V}\cong \cHom(\cS,\cQ)\cong \cS^\vee \otimes \cQ.
\]
The inner product on $\cO_{V}^{\oplus n}$ induces inner products on these other
bundles (using the canonical isomorphism
$\Lambda^*_\RR\cong \Lambda_\RR$ induced by the scalar product).
Thus it also induces a Riemannian metric.
All of these constructions work over $\Spec \ZZ$, giving rise to a smooth 
projective 
integral model $\Gr_\ZZ(m,n)$ together with bundles
$\cS_\ZZ$ (with $(\cS_{\ZZ})_\Lambda=\Lambda\subset\ZZ^n$),
$\cQ_\ZZ$ (with $(\cQ_{\ZZ})_\Lambda=\ZZ^n/\Lambda)$,
and $\cT_{V_\ZZ}=\cHom(\cS_\ZZ,\cQ_\ZZ)$. Note that
$\ZZ^n/\Lambda$ is isometric to $\Lambda^\pi$ via the orthogonal projection.
For a point $x=\Lambda\in \Gr(m,n)(\QQ)$, we are interested in the
\emph{tangent lattice}
\[
    T_\Lambda = (\cT_{V_\ZZ})_{x}
    =\cHom(\cS_\ZZ,\cQ_\ZZ)_x
    = \Hom(\Lambda,\ZZ^n/\Lambda) 
    \cong \Lambda^*\otimes_\ZZ \Lambda^\pi
\]
inside
\[
    \Lambda_\RR \otimes \Lambda_\RR^\perp
    \cong \Lambda_\RR^* \otimes \RR^n/\Lambda_\RR
    \cong (\cT_{V})_x\otimes \RR,
\]
where all isomorphisms are isometries.
Note that the dual of the tangent lattice is $T_\Lambda^*\cong \Lambda \otimes_\ZZ \Lambda^\perp$.
In view of~\eqref{eq:tensor-covolume}, 
this lattice has  covolume
\begin{equation}\label{eq:detT}
    \begin{split}
    \covol(T_\Lambda)
    &= \covol(\Lambda^*)^{n-m}\covol(\Lambda^\pi)^{m}\\
    &= \covol(\Lambda)^{m-n}\cdot \frac{1}{\covol(\Lambda)^{m}}\\
    &= \covol(\Lambda)^{-n}.
    \end{split}
\end{equation}
The logarithmic anticanonical height of $x=\Lambda$ is thus
\[
    h(x)
    = -\log\covol(\Lambda^*\otimes \Lambda^\pi)
    = n\log\covol(\Lambda).
\] 
The corresponding exponential height is
$
    H(x)=\covol(\Lambda)^{n}.
$

It follows from~\eqref{eq:borek} that there is an explicit constant $c_n>0$
such that 
\[
    0\leq \log s_{m(n-m)}(T_\Lambda)+ \mu_{\min}(T_\Lambda)\leq c_n,
\]
where $s_{m(n-m)}(T_\Lambda)$ is the largest successive minimum of the
tangent lattice. 
Hence, the definition~\eqref{eq:mu-free} yields
\begin{equation}\label{eq:f-P}
    \begin{split}
    \ell(x)
&=
    \frac{
        \max\left\{(\frac{m(n-m)}{n}) \mu_{\min}(T_\Lambda), 0\right\}
    }{
        \log \covol(\Lambda)
    }  \\  
    &=
    \frac{
        \max\left\{-(\frac{m(n-m)}{n}) \log s_{m(n-m)}(T_\Lambda), 0\right\}
    }{
        \log \covol(\Lambda)
    }+ O\left(\frac{1}{
        \log \covol(\Lambda)
    }
\right),
\end{split}
\end{equation}
for any primitive rank $m$ lattice $\Lambda\subset \ZZ^n$ representing a point
$x\in \gr(m,n)(\QQ)$.

One notes that 
\[
    s_{m(n-m)}(T_\Lambda)
    \geq \covol(T_\Lambda)^{\frac{1}{m(n-m)}}
    = \covol(\Lambda)^{-\frac{n}{m(n-m)}},
\]
by~\eqref{eq:upper-lower} and~\eqref{eq:detT}.
This argument suggests that if $x$ is ``typical'', in the sense that 
the successive minima of  $T_\Lambda$ all have equal order of magnitude, then
$\ell(x)= 1 +o(1)$, as $H(x)\to \infty$.

\subsection*{Projective space}

We proceed by discussing the  quantities we've introduced  in the most familiar case
$m=1$, 
corresponding to 
projective space $\PP_\QQ^{n-1}$.
The lattice $T_\Lambda$ then admits a particularly concrete description, as follows.

\begin{lemma}\label{lem:m=1}
    Let $x\in \gr(1,n)(\QQ)=\PP_\QQ^{n-1}(\QQ)$ be identified with 
    the lattice $\Lambda=\ZZ\x$, for 
    a primitive vector  $\x\in \ZZ^n$.
    Then
    $T_\Lambda$ is isometric to $\norm{\x}^{-1}( \ZZ^n\cap \x^\perp)^*$.
\end{lemma}
\begin{proof}
    We begin by noting that   $\covol(\Lambda)=\norm{\x}$. Let 
    us put $\Gamma_{\x}=\ZZ^{n}\cap\perpen \x$, which we note is
    a primitive lattice of rank $n-1$. We begin by showing that
    \begin{equation}\label{eq:G1}
        \factor{\Gamma_{\x}}=\norm \x^{-2}\perpen{\Gamma_{\x}}.
    \end{equation}
    Clearly, $\perpen{\Gamma_{\x}}=\ZZ \x$, which is a
    primitive lattice of rank one  contained in the  space $\RR \x$. We
    compute $\factor{\Gamma_{\x}}$ using the fact that
    $\factor{\Lambda}=\dual{\left(\perpen{\Lambda}\right)}$
    when $\Lambda$ is primitive. Recalling the definition of the dual lattice,
    we therefore obtain
    \begin{align*}
        \factor{\Gamma_{\x}}
        &=\dual{\left(\perpen{\Gamma_{\x}}\right)}\\\
        &=\left\{ \z \in \left(\perpen{\Gamma_{\x}}\right)_\RR :
            \left\langle \y,\z\right\rangle \in\ZZ
            \text{ for all $\y\in\perpen{\Gamma_{\x}}$}\right\} \\
        &=\left\{ t\mathbf{x} :
            t\in\RR \text{ such that }
            \left\langle t\mathbf{x}, m\mathbf{x} \right\rangle
            \in \ZZ \text{ for all $m\in\ZZ$} \right\}.
    \end{align*}
    But
    $\left\langle t\mathbf{x},m\mathbf{x}\right\rangle =
    tm \left\langle \mathbf{x},\mathbf{x}\right\rangle = tm\norm \x^2$.
    Moreover, $tm\norm \x^{2}\in\ZZ$ for every $m\in\ZZ$ if and only if
    $t\in \norm \x^{-2}\ZZ$. 
    We conclude that 
    \[
        \factor{\Gamma_{\x}}
        = \left\{ t\x:t\in\norm \x^{-2}\ZZ\right\}
        = \norm \x^{-2}\ZZ \x
        = \norm \x^{-2}\perpen{\Gamma_{\x}},
    \]
    as required for~\eqref{eq:G1}.

    Next we observe that 
    \[
        T_\Lambda = (\ZZ\x)^*\otimes_\ZZ (\ZZ\x)^\pi
        = \Gamma_\x^\pi \otimes_\ZZ\Gamma_\x^*.
    \]
    We wish to prove that this is isometric to 
    \[
        \norm \x^{-1}\dual{\Gamma_{\x}}.
    \]
    We know that
    $\factor{\Gamma_{\x}}
    \subset\left(\perpen{\Gamma_{\x}}\right)_\RR
    =\RR \x$
    and $\dual{\Gamma_{\x}}\subset \Gamma_{\x,\RR}=\perpen \x$,
    because the dual lattice always lives in the same space as the original lattice,
    and the factor lattice always lives in the orthogonal space. 
    Thus $T_{\Lambda}\subset\RR \x\otimes\perpen \x$.
    We claim that the map
    \begin{align*}
        \varphi\colon \RR \x\otimes\perpen \x & \to\perpen \x \\
        \al \x\otimes \w & \mapsto\al \norm \x \w
    \end{align*}
    is an isometry, where the inner product on the tensor product is the
    product of $\left\langle \cdot,\cdot\right\rangle $ in each of the
    components. 
    In the light of~\eqref{eq:G1}, 
    this will suffice to complete the proof of the lemma, since then
    $\varphi$ clearly maps the lattice
    $\factor{\Gamma_{\x}}\otimes_{\ZZ}\dual{\Gamma_{\x}}$
    to $\norm \x^{-1}\dual{\Gamma_{\x}}$. 
    To check the claim, it suffices to find an orthornormal basis of
    $\RR \x\otimes\perpen \x$ that is taken to an 
    orthonormal basis  of $\x^\perp$. Clearly, the orthonormal basis 
    \[
        \left\{ \norm \x^{-1}\x\otimes \w_{1},\dots,
        \norm \x^{-1}\x\otimes \w_{n-1}\right\} 
    \]
    suffices,
    where $\left\{ \w_{1},\dots,\w_{n-1}\right\} $ is an orthonormal basis
    for $\perpen \x$.
\end{proof}

We next consider what the definition~\eqref{eq:f-P} has to say when 
 $m=1$. Applying Lemma~\ref{lem:m=1}, we obtain $h(x)=n\log \norm{\x}$ and 
\[
\ell(x) =
\frac{n-1}{n} \cdot \max\left\{1-
\frac{ \log s_{n-1}(( \ZZ^n\cap \x^\perp)^*)}{\log \norm{\x} }, 0\right\}
+O\left(\frac{1}{h(x)}\right),
\]
if $x\in \PP_\QQ^{n-1}(\QQ)$ is represented by  the primitive vector $\x\in \ZZ^n$.
But
\[
    \log s_{n-1}(( \ZZ^n\cap \x^\perp)^*)=-\log s_{1}( \ZZ^n\cap \x^\perp)+O(1),
\]
by~\eqref{eq:succ-minima-duality}. Since $s_{1}(  \ZZ^n\cap \x^\perp)\geq 1$, this shows that  
\[
\ell(x)\geq \frac{n-1}{n}+o(1),
\]
as $h(x)\to \infty$, 
which essentially recovers~\eqref{eq:Pn-1}.

\subsection*{The general case}

Let $m\geq 1$ and let $\Lambda\in \gr(m,n)(\QQ)$. 
The largest successive minimum of the tangent lattice $T_\Lambda=\Lambda^*\otimes_\ZZ \Lambda^\pi$ is
closely related  to those of  $\Lambda^*$ and  $\Lambda^\pi$. Indeed, it follows from~\eqref{eq:s-tensor} that 
\[
    s_{m}(\Lambda^*)s_{n-m}(\Lambda^\pi)\ll 
    s_{m(n-m)}(T_\Lambda)\ll
    s_{m}(\Lambda^*)s_{n-m}(\Lambda^\pi),
\]
where the implied constants depend only on $n$.

Returning to~\eqref{eq:f-P},
there exists a constant $C>0$, which depends only on $n$, such that 
any given $x\in \gr(m,n)(\QQ)$ has  $\ell(x)<\ve+o(1)$ if and only if 
\begin{equation*}
s_{m(n-m)}(T_\Lambda)\geq C 
  \left(\covol(\Lambda)\right)^{-\frac{\ve n}{m(n-m)}}.
\end{equation*}
On redefining $C$, this is equivalent to 
\[
    s_{m}(\Lambda^*)s_{n-m}(\Lambda^\pi)
    \geq C  \left(\covol(\Lambda)\right)^{-\frac{\ve n}{m(n-m)}}.
\]
It now follows from~\eqref{eq:succ-minima-duality} that 
$\ell(x)<\ve+o(1)$ if and only if 
\begin{equation}\label{eq:small-s1}
    s_{1}(\Lambda^\perp)s_{1}(\Lambda)\leq C
    \left(\covol(\Lambda)\right)^{\frac{\ve n}{m(n-m)}},
\end{equation}
after a further modification to  $C$.

\begin{proof}[Proof of Theorem~\ref{t:unfree}]
    Let  $1<m<n-1$. In particular, it follows
    that $n\geq 4$.
    Our task is to show that there are infinitely many $x\in \gr(m,n)(\QQ)$
    such that $\ell(x)=0$.
    For this we work directly with the definition~\eqref{eq:f-P} of $\ell(x)$
    in terms of slopes. 
    Consider the two vectors $\uu=(q,1,0,\dots,0)\in \ZZ^n$ and
    $\mathbf{v}=(1,-q,0,\dots,0)\in \ZZ^n$, 
    for an arbitrary positive integer $q$.
    Let
    $\mathbf{e}_1,\dots,\mathbf{e}_n$ be the standard basis vectors of $\RR^n$.
    We take 
    \[
        \Lambda =
        \ZZ\uu \oplus \ZZ\mathbf{e}_3 \oplus \dots \oplus \ZZ \mathbf{e}_{m+1}.
    \]
    This is a primitive  lattice of rank $m$, with 
    $\covol(\Lambda)=q$ and $s_1(\Lambda)=1$.
    The orthogonal complement is 
    \[
        \Lambda^\perp =
        \ZZ\mathbf{v} \oplus \ZZ\mathbf{e}_{m+2} \oplus \dots
        \oplus \ZZ \mathbf{e}_{n}.
    \]
    This is a primitive lattice of rank $n-m$, with 
    $\covol(\Lambda^\perp)=q$ and $s_1(\Lambda^\perp)=1$.
    On appealing to~\eqref{eq:slope-duality}, it now follows that 
    \[
        \mu_{\min}(T_\Lambda) 
        = -\mu_{\max}(\Lambda\otimes_\ZZ \Lambda^\perp)
        \le \log \covol(\Span_\ZZ(\mathbf{e}_{m+1}\otimes \mathbf{e}_{n}))
        = 0,
    \]
    whence $\ell(x)=0$.
\end{proof}

Let $\Omega_\ve=\{\Lambda \in\gr(m,n)(\QQ) : \ell(x)\le \ve\}$ be the set of 
non-$\ve$-free points on $\gr(m,n)$. The remaining task for this section is to prove that this set is not thin,
as claimed in Theorem~\ref{t:unfree-not-thin}. 

\begin{lemma}
    If $\ve>0$ and $1<m<n-1$, then the image of $\Omega_\ve$ is dense in
    $\prod_{v} \gr(m,n)(\QQ_v)$.
\end{lemma}

\begin{proof}
    Let $S$ be a finite set of  places, and let
    $U_v \subset \gr(m,n)(\QQ_{v})$ be open subsets for $v\in S$.
    We want to find a $\Lambda\in\Omega_\ve$ whose image lies in
    $\prod_{v\in S} U_v$.
    By weak approximation, there exists $\Lambda_0 \in \gr(m,n)(\QQ)$
    whose image lies in
    $\prod_{v\in S} U_v$.
    Pick any rational line $l \subset \Lambda_{0,\RR} = \Lambda_0 \otimes \RR$
    and any rational hyperplane
    $H\subset \RR^n$ with
    $\Lambda_0\subset H$. We have $l=\langle \uu\rangle$ and $H=\langle\vv\rangle^\perp$
    for primitive $\uu, \vv\in\ZZ^n$.
    Consider the subvariety
    \[
        X=\{\Lambda \in \Gr(m,n): l \subset \Lambda_\RR \subset H\}
        \subset \Gr(m,n),
    \]
    which is isomorphic to $\Gr(m-1,n-2)$. Our 
    assumption 
    on $m$ ensures that  $\gr(m-1,n-2)$ is smooth and of positive dimension. Thus the non-empty
    open subsets $U_v \cap X(\QQ_{v})$ contain infinitely many
    $\QQ_{v}$-points for all $v\in S$.
    An application of weak approximation on $\Gr(m-1,n-2)$ therefore shows
    that there are infinitely many 
    $\Lambda^\prime \in X(\QQ)$ whose images lie in $U_v \cap X(\QQ_{v})$ for all 
    $v\in S$. 
    
    Let $C>0$ be sufficiently small that~\eqref{eq:small-s1} implies that the 
    corresponding point is not $\ve$-free.
    Since there are only finitely many points in $\gr(m-1,n-2)(\QQ)$ of bounded 
    height, we can find $\Lambda^\prime\in X(\QQ)$ whose images are in $U_v$
    for all 
    $v\in S$, with
    \[
        H(\Lambda^\prime) >
        \left(C^{-1} \norm{\uu}\norm{\vv}\right)^{\frac{m(n-m)}{\ve}}.
    \]
    Since the corresponding lattice $\Lambda$ satisfies
    $s_1(\Lambda)\le \norm{\uu}$ and $s_1(\Lambda^\perp)\le \norm{\vv}$,
    we have $P^\prime\in \Omega_\ve$.
\end{proof}

\begin{proof}[Proof of Theorem~\ref{t:unfree-not-thin}]
Appealing to work of Serre~\cite[Corollary~3.5.4]{Ser08}, the image of $\Omega_\ve$ would be
    nowhere dense if it where thin.
\end{proof}

\begin{remark}
    We note that $\Omega_0$ is always a thin set, since it fails to be Zariski dense. Indeed, when 
    $1<m<n-1$, we find that $\Omega_0$
    is the union of the proper subvarieties $\{P\in \gr(m,n)(\QQ): \uu\in P\}$, 
    where $\uu$ runs over the finitely many 
    vectors in $\ZZ^n$ of norm $O(1)$.
    If, on the other hand, $m$ is $1$ or $n-1$, then $\Omega_\ve=\emptyset$
    for $\ve<\frac{n-1}{n}$ by~\eqref{eq:Pn-1}. 
\end{remark}

\section{Free rational points dominate}\label{s:free-pf}

In this section we give our first proof of Theorem~\ref{t:free}.
We must provide an upper bound for the quantity
\[
E_{\ve}(B)=\#\{x\in \gr(m,n)(\QQ): H(x)\leq B, ~\ell(x)<\ve\},
\]
with the aim being to show that  $E_\ve(B)=o(B)$
for any  $0\leq \ve<1$. 
Switching to the language of primitive lattices, 
we write $P(m,n)$ for the set of primitive lattices $\Lambda\subset \ZZ^n$ which have rank $m$.
It then  follows from~\eqref{eq:small-s1} that $E_{\ve}(B)$ is at most 
\begin{align*}
\#\left\{ \Lambda \in P(m,n): \covol(\Lambda)\leq B^{ \frac{1}{n} }, ~
s_{1}(\Lambda^\perp)s_{1}(\Lambda)\leq C
 \left(\covol(\Lambda)\right)^{\frac{\ve n}{m(n-m)}}
\right\},
\end{align*}
for a suitable constant $C>0$ depending only on $n$.
It is convenient to 
break the range for the covolume and the sizes of $s_1(\Lambda)$ and $s_1(\Lambda^\perp)$
into  dyadic intervals. Thus we put 
\begin{equation}\label{eq:B-R}
E_{\ve}(B)
\leq \sum_{\substack{R=2^j\\1\leq R\leq 2B^{1/n}}}
\sum_{\substack{S_1=2^{j_1}, ~S_2=2^{j_2}\\1\leq S_1S_2\leq 
CR^{\frac{\ve n}{m(n-m)}}}}
\widetilde E_{\ve}(R,S_1,S_2),
\end{equation}
where
\[
    \widetilde E_{\ve}(R,S_1,S_2)=
    \#\left\{
    \Lambda \in P(m,n): 
    \begin{array}{l}
    R/2<\covol(\Lambda)\leq R,~\\
    S_1/2<s_{1}(\Lambda)\leq S_1,~\\
   S_2/2<s_{1}(\Lambda^\perp)\leq S_2
       \end{array}
    \right\}.
\]
Let 
\[
P(m,n;R,S)= \left\{ \Lambda \in P(m,n): \covol(\Lambda)\leq R, ~
s_{1}(\Lambda)\leq S\right\}.
\]
Then, since 
the 
 covolume is preserved under taking the  orthogonal complement, we see that 
 \begin{equation}\label{eq:bound-E-eps}
  \widetilde E_{\ve}(R,S_1,S_2)\leq 
  \min\left\{
 \#P(m,n;R,S_1), ~  \#P(n-m,n;R,S_2)
  \right\}.
 \end{equation}
Our attention now shifts to estimating 
$\#P(m,n;R,S)$
for given $1\leq m<n$ and $R,S\geq 1$.

Let $r\geq 1$ be an integer and let 
$s_1,\dots,s_r\in \RR$ with 
\[
1\leq s_1\leq s_2\leq \cdots \leq s_r.
\]
We let $P_{r,n}(s_1,\dots,s_r)$ be the set of primitive  lattices $\Lambda \subset \mathbb{Z}^n$  
of rank $r$
whose $i$th successive minimum lies in the interval $[s_i,2s_i)$, for $1\leq i\leq r$. We shall need the following result. 

\begin{lemma}\label{lem:count}
We have
\[
\#P_{r,n}(s_1,\dots,s_r)\ll s_1^{n+r-1}s_2^{n+r-3}\cdots s_r^{n+1-r},
\]
where the implied constant only depends on $n$. 
\end{lemma}

\begin{proof}
This is extracted from Lemma 6 of 
Schmidt~\cite{schmidt} 
in work of Browning, Le Boudec and Sawin~\cite[Lemma 3.18]{random}.
\end{proof}

The following general inequality will facilitate our application of 
Lemma~\ref{lem:count}.
Let $r\geq 1$ and let 
$\xi_1,\dots,\xi_r, \alpha_1,\dots,\alpha_r\geq 0$. 
Then 
\begin{equation}\label{eq:lin-prog}
\xi_1^{\alpha_1}\cdots \xi_r^{\alpha_r}
\leq \left(\xi_1 \cdots \xi_r\right)^{\frac{\alpha_1+\cdots +\alpha_r}{r}},
\end{equation}
provided that $1\leq \xi_1\leq \cdots\leq \xi_r$ and $\alpha_1\geq \cdots \geq \alpha_r$.
This is easily proved by induction on $r$. Indeed, on putting $\beta=\alpha_1+\cdots+\alpha_{r-1}$, the induction hypothesis yields
\begin{align*}
\xi_1^{\alpha_1}\cdots \xi_r^{\alpha_r}
& \leq \left(\xi_1 \cdots \xi_{r-1}\right)^{ \frac{\beta}{r-1} }
        \cdot \xi_r^{\alpha_r}\\
& \leq \left(\xi_1 \cdots \xi_r\right)^{ \frac{\beta+\alpha_r}{r} }
        \cdot \left(\xi_1 \cdots \xi_{r-1}\right)^{ \frac{\beta}{r(r-1)}-\frac{\alpha_r}{r} }
        \cdot \xi_r^{ \alpha_r-\frac{\beta+\alpha_r}{r} }.
\end{align*}
Noting that $\beta/(r-1)\geq \alpha_r$,  we may take  $\xi_1,\dots, \xi_{r-1}\leq \xi_r$ in the second factor and then immediately 
arrive at the right hand side of~\eqref{eq:lin-prog}.

We are interested in the case $r=m$. Breaking into dyadic intervals, we see that 
\[
\#P(m,n;R,S)
\leq \sum_{\substack{
1\leq s_1 \leq \cdots \leq s_m\ll R\\
s_1\dots s_m \ll R\\
s_1\leq S}} \#P_{m,n}(s_1,\dots,s_m),
\]
where $s_1,\dots,s_m$ run over powers of $2$, subject to the stated inequalities.  
But Lemma~\ref{lem:count} yields
\begin{align*}
\#P_{m,n}(s_1,\dots,s_m)
&\ll s_1^{n+m-1} s_2^{n+m-3}\cdots s_m^{n+1-m}.
\end{align*}
Hence it follows from~\eqref{eq:lin-prog} that
\[
s_2^{n+m-3} s_3^{n+m-5}\cdots s_m^{n+1-m}
\ll 
\left(s_2\cdots s_m\right)^{\frac{\Delta}{m-1}},
\]
where
\begin{align*}
\Delta&=\sum_{k=0}^{m-2} (n+1-m+2k)=(m-1)(n-1).
\end{align*}
Thus
\begin{align*}
    \#P(m,n;R,S)
    &\ll \sum_{\substack{
        1\leq s_1 \leq \cdots \leq s_m\ll R\\
        s_1\cdots s_m \ll R\\ s_1\leq S
    }}
    s_1^{n+m-1} \left(s_2\cdots s_m\right)^{n-1}\\
    &\ll R^{n-1} (\log R)^{m-2} \sum_{1\leq s_1\leq S} \frac{s_1^{n+m-1}}{s_1^{n-1}} \\
   &\ll S_1^m R^{n-1} (\log R)^{m-2}.
\end{align*}

Returning to~\eqref{eq:bound-E-eps}, we 
apply the inequality $\min\{\alpha,\beta\}\leq \alpha^{\frac{n-m}{n}}\beta^{\frac{m}{n}}$ for any $\alpha,\beta\in \RR_{\geq 0}$, in order to  deduce that 
\begin{align*}
  \widetilde E_{\ve}(R,S_1,S_2)
 & \ll
  \left(S_1^{m}R^{n-1}  (\log R)^{m-2}\right)^{\frac{n-m}{n}} 
  \cdot
  \left(S_2^{n-m}R^{n-1}  (\log R)^{n-m-2}\right)^{\frac{m}{n}} \\
      &\leq  (S_1S_2)^{\frac{m(n-m)}{n}}R^{n-1}  (\log R)^{2m-2}.
\end{align*}
But then, 
on returning to~\eqref{eq:B-R} and  summing over dyadic intervals for $R, S_1,S_2$, we obtain  
\begin{align*}
E_{\ve}(B)
&\ll  \sum_{\substack{R=2^j\\1\leq R\leq 2B^{1/n}}}
R^{n-1+\ve}  (\log R)^{2m-2}\\
&\ll  
B^{1-\frac{1-\ve}{n}}  (\log B)^{2m-2}.
\end{align*}
This gives an explicit power saving error term if and only if $\ve<1$.
This therefore completes the proof of Theorem~\ref{t:free}.

\section{Tangent lattices equidistribute}\label{s:equidistribution}

The goal of this 
 section is to prove Theorem~\ref{t:equi}, namely to establish equidistribution of the projections of the
tangent lattices $\dual{\Lambda}\otimes_{\ZZ}\factor{\Lambda}$ in $\cUG_{m,n}$
with respect to 
a certain probability measure $\nu$ which will be defined below.

\subsection*{Spaces of lattices}
We begin by describing the spaces $\cL_{m,n}$ and $\cG_{m,n}$ of lattices as quotients.
Recalling the space $\cL_{m,n}$ defined in~\eqref{eq:Lmn}, we note that it is isomorphic to 
\[
\cL_{m,n}\cong\SL_n(\RR)/\left(\left[\begin{smallmatrix}
\GL_m(\ZZ) & \Mat_{m,n-m}(\RR)\\
0_{n-m,m} & \GL_{n-m}(\RR)
\end{smallmatrix}\right]\cap\SL_n(\RR)\right).
\]
For our needs, however, it will be useful to view it as the non-homogeneous
quotient $\cL_{m,n}\cong\indMat{n\times m}(\RR)/\GL_m(\ZZ)$,
where $\indMat{n\times m}(\RR)$ is the open subset of $\Mat_{n\times m}(\RR)$
consisting of matrices of full rank.
 
With the latter point of view, the subspace $\cG_{m,n}$ can be described as a quotient
$Y/\GL_{n^2}(\ZZ) \subset \indMat{n^2\times m(n-m)}(\RR)/\GL_{n^2}(\ZZ)$, where 
\begin{equation*}
    Y=	\left\{A\otimes  B : 
    \begin{array}{l}
    A\in \Mat_{n\times m}(\RR), 
    ~B\in \Mat_{n\times (n-m)}(\RR)\\ 
    \Span_\RR(A)=\Span_\RR(B)^\perp
    \end{array}
    \right\},
\end{equation*}
and where we interpret $A\otimes  B$ as $(a_{i,j}B)_{i,j}$ for  $A=(a_{i,j})\in \Mat_{n\times m}(\RR)$.
(Note that the conditions in $Y$ imply that $A$ and $B$ must be of full rank.)

Next, we extend the notion of a factor lattice to lattices that are not necessarily integral: Recall that a rank $m$ lattice $\Lambda$ is 
\emph{primitive inside a full unimodular lattice} $\Gamma$ if there is no rank $m$ sublattice of $\Gamma$ that properly contains $\Lambda$.
(Note that when $m<n$, every lattice is primitive with respect to infinitely many full unimodular lattices in $\RR^{n}$.)
For each  full unimodular lattice $\Gamma$ in which $\Lambda$ is primitive, we define the \emph{factor lattice of $\Lambda$ with respect to $\Gamma$} as $\factorg{\Lambda}{\Gamma}=\pi(\Gamma)$, where $\pi$ is
the orthogonal projection from $\RR^{n}$ to $\perpen{(\Span_{\RR}(\Lambda))}$. We have 
\[
    \covol(\factorg{\Lambda}{\Gamma})=\frac{1}{\covol(\Lambda)}.
\] 
In particular, $\factorg{\Lambda}{\Gamma}$ is unimodular if and only if  $\Lambda$ is. 

Consider the space of pairs 
\[
    \cP_{m,n}
        =
    \left\{(\Lambda,\factorg{\Lambda}{\Gamma}) : \text{$\Lambda$ primitive of rank $m$ in 
           a full unimodular lattice $\Gamma\subset \RR^{n}$  }\right\}.
\]
We observe that
\[
    \cP_{m,n}
    \cong\SL_n(\RR)/\left\{ \left[\begin{smallmatrix}
        \GL_m(\ZZ) & \Mat_{m,n-m}(\RR)\\
        0_{n-m,m} & \GL_{n-m}(\ZZ)
    \end{smallmatrix}\right]
    \cap\SL_n(\RR)\right\}.
\]
To see this, note that the first $m$ columns of any matrix in a given coset span the same lattice $\Lambda$, and the projections of the 
last $n-m$ columns of such a matrix span
$\factorg{\Lambda}{\Gamma}$, where $\Gamma$ is the full lattice spanned by the
columns of the matrix. The space $\cP_{m,n}$ is also a $\PGL_n(\RR)$-homogeneous
space since 
\[
    \cP_{m,n}\cong\PGL_n(\RR)/\left\{ \left[\begin{smallmatrix}
        \GL_m(\ZZ) & \Mat_{m,n-m}(\RR)\\
        0_{n-m,m} & \GL_{n-m}(\ZZ)
    \end{smallmatrix}\right]\right\}.
\]

\subsection*{Subspaces of unimodular lattices}

For the sake of defining the measure $\nu$ appearing in Theorem~\ref{t:equi}
and of proving Proposition~\ref{prop:diffeo} below, we briefly discuss the
subsets of unimodular elements inside $\cL_{m,n}$, $\cG_{m,n}$, 
and $\cP_{m,n}$. Let $\cUL_{m,n}$ denote the subset of unimodular lattices in
$\cL_{m,n}$, and observe that $\cUL_{m,n}\cong\cL_{m,n}/D$, where 
\[
    D =
    \{\diag{\al^{-\frac{1}{m}}\idmat m,\al^{\frac{1}{n-m}}\idmat{n-m}} :
        \alpha\in\RR_{>0}
    \}
    \cong\RR_{>0}.
\]
This is a one-parameter subgroup of diagonal matrices in $\SL_n(\RR)$.
Moreover, $\cUL_{m,n}\cong\induniMat{n\times m}(\RR)/\GL_m(\ZZ)$,
where $\induniMat{n\times m}(\RR)$ is the set of matrices $M$
of full rank that satisfy $\det(M^{\transpose}M)=1$. 

Recall that $\cUG_{m,n} = \cG_{m,n} \cap \cUL_{m,n}$.
Defining
\[
    \uni Y =  \{ A\otimes  B\in Y : \det(A)^{n-m}\det(B)^m=1 \},
\]
we get
\[
    \cUG_{m,n} \cong \uni {Y}/\GL_{m(n-m)}(\ZZ).
\]

Let $\cUP_{m,n}$ be the subset of pairs in $\cP_{m,n}$
for which $\Lambda$ (and therefore also $\factorg{\Lambda}{\Gamma}$)
is unimodular. 

The space $\cL_{m,n}$ decomposes as
\[
    \cL_{m,n}\cong\cUL_{m,n}\times\RR_{>0}
\] 
via $\Lambda\mapsto(\covol(\Lambda)^{-1/m}\Lambda,\covol(\Lambda))$,
with inverse map $\alpha^{1/m}L\mapsfrom(L,\alpha)$. 

For $\cG_{m,n}$, we use the decomposition
\[
    \cG_{m,n} \cong \cUG_{m,n}\times\RR_{>0},
\]
given by 
\[
    \Lambda \mapsto
    \left(\covol(\Lambda)^{-\frac{1}{m(n-m)}}\Lambda,
    \covol(\Lambda)^{-1}
    \right),
\]
with the inverse map being given by
$\al^{-1/(m(n-m))} L \mapsfrom(L,\al)$.

Lastly, $\cP_{m,n}$ decomposes as
\[
    \cP_{m,n}\cong\cUP_{m,n}\times\RR_{>0}
\] 
via 
\[
    (\Lambda,\factorg{\Lambda}{\Gamma})
    \mapsto(\covol(\Lambda)^{ -\frac{1}{m} } \Lambda, 
            \covol(\Lambda)^{\frac{1}{n-m}}\factorg{\Lambda}{\Gamma},\covol(\Lambda)),
\]
the inverse map being given by 
\[
    (\al^{\frac{1}{m}} L,\alpha^{-\frac{1}{n-m}}\factorg{L}{\Gamma^\prime})
    \mapsfrom(L,\factorg{L}{\Gamma^\prime},\al).
\]
In all three cases, we denote by $u$ the projection to the unimodular component. In particular, $\cG_{m,n} \to \cUG_{m,n} \times \RR_{>0}$ maps $T_\Lambda\mapsto (u(T_\Lambda), H(\Lambda))$, for $\Lambda\in\Gr(m,n)$.

The decomposition of $\cP_{m,n}$ extends to a decomposition of
homogeneous measure spaces as follows.
Let $\vol_{\cP_{m,n}}$ be an $\SL_n$-invariant measure on 
$\cP_{m,n}$ (as described, for example, in~\cite[Thm.~2.51]{Fol15}). 
Then the  Iwasawa decomposition of the Haar measure on $\mathrm{SL}_n$ yields
\begin{align*}
    (\cP_{m,n},\vol_{\cP_{m,n}})
    \cong(\cUP_{m,n},\vol_{\cUP_{m,n}})\times(\RR_{>0},x^{n-1} \diff x),
\end{align*}
for a finite measure $\vol_{\cUP_{m,n}}$ on $\cUP_{m,n}$, and we normalise all measures so that this measure is a probability measure.

\subsection*{From pairs $(\Lambda, \factor{\Lambda})$ to tensors $\dual{\Lambda} \otimes_\ZZ \factor{\Lambda}$}

We now summarise our strategy for proving Theorem~\ref{t:equi}.
We construct a diffeomorphism $\uni{\phi}\colon \cUP_{m,n}\to \cUG_{m,n}$ that will enable us to apply an equidistribution
statement for the primitive lattices in $\cUP_{m,n}$ with
respect to the measure $\vol_{\cUP_{m,n}}$.
The measure in our equidistribution theorem will be the pushforward $\nu=\uni{\phi}_*\vol_{\cUP_{m,n}}$ of $\vol_{\cUP_{m,n}}$.

\begin{proposition}\label{prop:diffeo}
    The map
    \[
        \phi \colon \cP_{m,n}\to \cG_{m,n}
    \]
    sending $(\Lambda,\factorg{\Lambda}{\Gamma})\mapsto\dual{\Lambda}\otimes_\ZZ \factorg{\Lambda}{\Gamma}$
    is a diffeomorphism, as is its restriction 
    \[
        \uni{\phi}\colon \cUP_{m,n}\to \cUG_{m,n}
    \]
    to the unimodular elements. Moreover, we have $\phi=\uni{\phi}\times (\,\cdot\,)^n$.
\end{proposition}
Let
\[
    U=\left\{ \left[\begin{smallmatrix}
    \idmat m & \Mat_{m,n-m}(\RR)\\
    0 & \idmat{n-m}
    \end{smallmatrix}\right]\right\},
\] 
so that
\[
    \cP_{m,n}=\PGL_n(\RR)/\left(U\rtimes\left[\begin{smallmatrix}
        \GL_m(\ZZ) & 0\\
        0 & \GL_{n-m}(\ZZ)
    \end{smallmatrix}\right]\right).
\]
Clearly, $\dual{\Lambda}\otimes_\ZZ\factor{\Lambda}$ is a lattice of rank
$\rank(\dual{\Lambda})\cdot\rank(\factor{\Lambda})=m(n-m)$ inside the space
$\RR^{n}\otimes\RR^{n}\cong\RR^{n^{2}}$. 
The diffeomorphism will
be constructed in three steps:
\begin{enumerate}[label=(\roman*)]
    \item\label{enum:step-tilde-phi} define a differentiable map $\tilde{\phi}\colon \GL_n(\RR)\to Y$;
    \item\label{enum:step-phi-hat} obtain a diffeomorphism $\hat \phi\colon \PGL_n(\RR)/U\to Y$; and
    \item\label{enum:step-phi} reduce $\hat \phi$ modulo the integral subgroups to $\phi\colon \cP_{m,n}\to \cG_{m,n}$.
\end{enumerate}

\subsubsection*{Step~\ref{enum:step-tilde-phi}}

Let $\tilde{\phi}\colon \GL_n(\RR)\to Y$ be the map which is defined as
follows. For $g\in\GL_n(\RR)$, let $A_{g}\in\indMat{n\times m}(\RR)$
be the matrix consisting of the first $m$ columns of $g$.
Let $\tilde{A}_{g}
=A_{g}(A_{g}^{\transpose}A_{g})^{-1}\in\indMat{n\times m}(\RR)$, 
and let $\tilde{B}_{g}\in\indMat{n\times (n-m)}(\RR)$ be the matrix obtained from
projecting the columns of $g$ to the orthogonal subspace of $\tilde{A}_{g}$
(or $A_{g}$, as they span the same space).  Then we define
\[
    \tilde{\phi}(g)=\tilde A_{g}\otimes \tilde B_{g},
\]
which clearly lies in $Y$.
We note that $\tilde{\phi}$ is differentiable since it is a composition of differentiable maps.
Moreover, if $g=(A|B)$ generates a lattice $\Gamma$ and $A$ generates a lattice $\Lambda$,
then $\tilde\phi(g)$ generates $\Lambda^* \otimes_\ZZ \factorg{\Lambda}{\Gamma}$. 

\subsubsection*{Step~\ref{enum:step-phi-hat}}
By the definition of $\tilde B_g$, the map $\tilde \phi$
is invariant under the unipotent subgroup $U$. On the other hand,
multiplying $g$ by a constant $\lm$ changes $\tilde A_{g}$ by $\lm^{-1}$
and $\tilde B_{g}$ by $\lm$, so $\tilde{\phi}$ is also invariant under
the center of $\GL_n$. Hence $\tilde{\phi}$ descends to the quotient
$X
=\PGL_n(\RR)/U$, and we denote this map by $\hat\phi\colon X\to Y$. 

\begin{lemma}
    $\hat\phi\colon X\to Y$ is a diffeomorphism. 
\end{lemma}

\begin{proof}
    For $C\in Y$, write $C=A\otimes B$ with $A\in \indMat{n\times m}(\RR)$ and $B\in\indMat{n\times (n-m)}(\RR)$. For $\tilde{A}=A(A^{\transpose}A)^{-1}$, let
    $\psi\colon Y\to\PGL_n(\RR)$ be the map $C\mapsto(\tilde{A}|B)$.
    Indeed, $(\tilde{A}|B)$ has full rank by the assumptions on $A$ and $B$, 
    and it is independent of the choice of $A$ and $B$ as any other representation 
    $C = \lambda^{-1}A\otimes \lambda B$, for some $\lambda \in \RR^\times$, leads to
    \[
        (\widetilde{\lambda^{-1}A} | \lambda B) = \lambda (\tilde A|B),
    \]
    which has the same class as $(\tilde A|B)$ in $\PGL_n(\RR)$. 

    We claim that the map $\hat\psi\colon Y\to X$ obtained by reducing $\psi$
    mod $U$ is the inverse of $\hat\phi$.
    Note that the matrix $\psi\circ\hat\phi(A|B)$
    differs from $(A|B)$ by adding multiples of the first columns to
    the last columns, i.e., by an element of $U$, so
    $\hat\psi\circ\hat\phi=\id_{X}$.
    Conversely, the tensor $\hat\phi\circ\hat\psi(A\otimes B)$ is again
    $A\otimes B$, since $\tilde{\tilde{A}}=A$, and $A$, $B$ are already
    orthogonal, so the orthogonalisation step in $\hat\phi$ changes
    nothing. Thus $\hat\phi\circ\hat\psi=\id_{Y}$. Finally, $\hat\phi$
    is a differentiable since $\tilde{\phi}$ is, and $\hat\psi$ is
    differentiable because $\psi$ is. 
\end{proof}

\subsubsection*{Step~\ref{enum:step-phi}}
We first note that $\hat \phi$ descends to the quotient. Indeed, for a matrix $(A|B)\in\GL_n(\RR)$
and $(\ga_{m},\ga_{n-m})\in\GL_m(\ZZ)\times\GL_{n-m}(\ZZ)$, we have
\[
    \hat \phi\colon \left(A|B\right)(\ga_{m},\ga_{n-m})
    \mapsto(A\ga_{m}\otimes B\ga_{n-m})=(A\otimes B)(\ga_m\otimes\ga_{n-m}),
\]
so $(A|B)(\GL_m(\ZZ)\times\GL_{n-m}(\ZZ))$ maps into $(A\otimes B)\GL_{m(n-m)}(\ZZ)$.
Also note that $\phi$ is surjective since $\hat\phi$ is. The crucial
part is then to show injectivity, which we do using two technical lemmas.

\begin{lemma}\label{lem:elementary-tensors}
    Let $A\otimes B\in Y$. If $g\in\GL_{m(n-m)}(\RR)$ is such that $(A\otimes B)g\in Y$,
    then $g$ is an elementary tensor. 
\end{lemma}

\begin{proof}
    Write $(A\otimes B)g=A^{\prime}\otimes B^{\prime}\in Y$.
    Since $g\in\GL_{m(n-m)}(\RR)$, it follows that
    $\Span(A\otimes B)=\Span(A^{\prime}\otimes B^{\prime})$,
    which means that
    $\Span(A)\otimes\Span(B)=\Span(A^{\prime})\otimes\Span(B^{\prime})$.
    As a result, $\Span(A)=\Span(A^{\prime})$ and $\Span(B)=\Span(B^{\prime})$,
    and therefore there exist $g_{m}\in\GL_m(\RR)$ and $g_{n-m}\in\GL_{n-m}(\RR)$
    such that $A^{\prime}=Ag_{m}$ and $B^{\prime}=Bg_{n-m}$.
    Thus $(A\otimes B)(g-g_{m}\otimes g_{n-m})=0$,
    which implies that $g=g_{m}\otimes g_{n-m}$,
    since $A\otimes B$ has full rank.
\end{proof}

\begin{lemma}\label{lem:Stab(Y)}
    Let $A\otimes B\in Y$, and let $\ga\in\GL_{m(n-m)}(\ZZ)$
    such that $(A\otimes B)\ga\in Y$. Then $\ga=\ga_{m}\otimes\ga_{n-m}$
    for $\ga_{m}\in\GL_m(\ZZ)$ and $\ga_{n-m}\in\GL_{n-m}(\ZZ)$.
\end{lemma}

\begin{proof}
    Let $A\otimes B$ and $\ga$ be as in the statement.
    Appealing to Lemma~\ref{lem:elementary-tensors}
    and using the fact that $\Mat_{m(n-m)\times m(n-m)}(\ZZ)=\Mat_{m\times m}(\ZZ)\otimes\Mat_{(n-m)\times(n-m)}(\ZZ)$,
    we have $\ga=\ga_{m}\otimes\ga_{n-m}$, with $\ga_{i}\in\GL_i(\RR)\cap\Mat_{i\times i}(\ZZ)$
    for $i\in \{m,n-m\}$.  Since 
    \[
        \det(\ga_{m})^{n-m}\det(\ga_{n-m})^{m}=\det(\ga_{m}\otimes\ga_{n-m})=\det(\ga)\in\{\pm 1\},
    \]
    it follows that  $\det(\ga_{m}), \det(\ga_{n-m})\in \{\pm 1\}$.
    We conclude that  $\ga_{m}\in\GL_m(\ZZ)$ and $\ga_{n-m}\in\GL_{n-m}(\ZZ)$.
\end{proof}

\begin{proof}[Proof of Proposition~\ref{prop:diffeo}]
    As already noted, $\phi$ is well-defined and surjective. It is also
    injective by Lemma~\ref{lem:Stab(Y)}, since if two elements in
    $Y$ are equivalent modulo $\GL_{m(n-m)}(\ZZ)$, then their pre-images
    are equivalent modulo $\GL_m(\ZZ)\times\GL_{n-m}(\ZZ)$. Thus $\phi$
    is also bijective, and therefore the inverse map of $\hat\phi$ descends
    to an inverse map of $\phi$. Since $\hat\phi$ and its inverse are
    differentiable, so are $\phi$ and its inverse. 

    Composing $\phi$ with the decompositions of $\cP_{m,n}$ and $\cG_{m,n}$
    results in
    \begin{equation*}
        \begin{tikzcd}[cramped, sep = small]
            \cUP_{m,n}\times\RR_{>0} \arrow[r] &
            \cP_{m,n} \arrow[r, "\phi"] &
            \cG_{m,n} \arrow[r] &
            \cUG_{m,n} \times \RR_{>0} \\
            (L,\factorg{L}{\Gamma},\al) \arrow[r, mapsto] &
            (\al^{\frac{1}{m}}L,\al^{-\frac{1}{n-m}}\factorg{L}{\Gamma}) \arrow[r, mapsto] & 
            \al^{-\frac{n}{m(n-m)}}\dual L \otimes_\ZZ \factorg{L}{\Gamma} \arrow[r, mapsto]&
            (\dual L\otimes_\ZZ\factorg{L}{\Gamma},\al^n),
        \end{tikzcd}
    \end{equation*}
    hence $\phi$ restricts to a diffeomorphism
    $\uni\phi \colon \cUP_{m,n} \to \cUG_{m,n}$ that satisfies
    $\phi = \uni\phi \times (\cdot)^n$, as required.
\end{proof}

\subsection*{The main equidistribution theorem}

The equidistribution statement in Theorem~\ref{t:equi} employs
counting in subsets of the spaces $\cP_{m,n}$ and $\cUP_{m,n}$
satisfying the following boundary condition.

\begin{definition}
    The (topological) boundary of a set inside a manifold $\mathcal{M}$
    is said to be \textit{controlled} if it consists of finitely many subsets of
    embedded $C^{1}$ submanifolds whose dimension is strictly smaller
    than $\dim\mathcal{M}$.
\end{definition}

Having proved that the space $\cP_{m,n}$ is diffeomorphic to
$\cG_{m,n}$ via $\phi$, we may apply the following counting result for pairs in
$\cP_{m,n}$, in order to deduce our main equidistribution result about tangent lattices.

\begin{theorem}[Horesh--Karasik~\cite{HK1}]\label{thm:counting-pairs}
    Let $m,n$ be integers such that $1\leq m\leq n-1$. Let $E \subseteq\cUP_{m,n}$ be a subset
    whose  boundary is  controlled. 
    For any $B\geq 1$, let $E_B = E \times (0,B]\subset\cP_{m,n}$ denote the
    subset of pairs $(\Lambda,\factorg{\Lambda}{\Gamma})$ that project to $E$ and for
    which $\covol(\Lambda)\leq B$. 
    Then
    \[
        \#\left\{(\Lambda,\factor{\Lambda})\in E_B :
                 \text{$\Lambda$ primitive}\right\}
        = c_{m,n} \vol_{\cUP_{m,n}}(E)
        B^{n} + O_{\ve,E}(B^{n(1-\dl_{E})+\ve}),
    \] 
    for any $\ve>0$, where $c_{m,n}$ was defined in~\eqref{eq:cmn}
    and
    \[
        \dl_{E} =\begin{cases}
            \frac{\left\lceil \frac{n-1}{2}\right\rceil}{4n^{2}}
            & \text{ if $E$ is bounded,}\\
            \frac{\left\lceil \frac{n-1}{2}\right\rceil}{4n^{2}\cdot[2(\max(m,n-m)-1)(n^2-1)+n^2]} 
            & \text{ if $E$ is not bounded.}
        \end{cases}
    \]
\end{theorem}

To be precise, this  result follows on combining~\cite[Thm.~3.2(4)]{HK1}
with the discussion in~\cite[Remark~3.3]{HK1} around
normalising the invariant measures to get probability measures, thus obtaining the leading constant.
We are now ready to prove the following result, which is a more precise version of 
Theorem~\ref{t:equi}. 

\begin{theorem}\label{t:equi'}
    Let $m,n$ be integers such that $1\leq m\leq n-1$. Let $E \subseteq\cUG_{m,n}$ be a subset
    whose  boundary is  controlled. 
Define
    $
        E_{B} = E \times (0,B] \subset \cG_{m,n},
    $
    for $B\geq1$. Then  
    \begin{align*}
        \frac{
            \#\{ \Lambda \in \Gr(m,n)(\QQ) : T_\Lambda \in E_{B}\}
        }{
            \#\{\Lambda \in \Gr(m,n)(\QQ) : H(\Lambda) \le B \}
        }
        & = \nu(E)
        + O_{E}(B^{-\frac{1}{16n^4}}).
    \end{align*}
\end{theorem}

\begin{proof} 
    We claim that 
    \[
        \delta_{E}>\frac{1}{16n^4}
    \]
    in the statement of Theorem~\ref{thm:counting-pairs}. This is obvious when
    $E$ is bounded. Alternatively, when 
    $E$ is not bounded, we note that
    \[
        2(\max(m,n-m)-1)(n^2-1)+n^2\leq 2(n-2)(n^2-1)+n^2<2(n-1)n^2,
    \]
    for any $1\leq m\leq n-1$ and $n\geq 2$. But then it follows that 
    \[
        \delta_{E}>\frac{\frac{1}{2}(n-1)}{4n^2\cdot 2(n-1)n^2}=\frac{1}{16n^4},
    \]
    as claimed.

    The set $E_B$ consists of all the lattices $L\in \cG$ whose
    normalisation lies in $E$ and for which
    $\covol(L)^{-1}\leq B$.
    It contains precisely those $T_\Lambda$ for primitive
    $\Lambda$ that project to $E$ and have $H(\Lambda)\leq B$.
    Since $\phi$ defines a bijection between the pairs $(\Lambda,\factor{\Lambda})$,
    where $\Lambda$ is primitive, and the elements
    $T_\Lambda = \dual{\Lambda}\otimes_\ZZ\factor{\Lambda}$,
    where $\Lambda$ is primitive, we deduce that the number of tangent lattices in
    $E_B$ is in fact equal to the number of
    ``primitive'' pairs  in $\phi^{-1}(E_{B})$.
    The latter can be estimated by Theorem~\ref{thm:counting-pairs}
    since  Proposition~\ref{prop:diffeo} implies that
    $\phi^{-1}(E_{B})=(\uni{\phi})^{-1}(E)\times(0,B^{1/n}]$,
    where $\uni{\phi}$ is a diffeomorphism, so that in particular
    $(\uni{\phi})^{-1}(E)$ has controlled boundary.
    On taking a sufficiently small choice of $\ve$ in 
    Theorem~\ref{thm:counting-pairs}, an application of this result 
    now yields
    \begin{align*}
        \#\{\Lambda \in \Gr(m,n)(\QQ) : T_\Lambda \in E_{B}\}
        &=
        c_{m,n}  \vol_{\cUP_{m,n}}((\uni{\phi})^{-1}(E))  B
        + O_{E} \left( B^{1-\frac{1}{16n^4}} \right)\\ 
        &= c_{m,n} \nu(E)
        B
        + O_{E}\left(B^{1-\frac{1}{16n^4}}\right). 
    \end{align*}
    The statement of the theorem now follows with the asymptotic expression~\eqref{eq:schmidt}
    for the denominator.
\end{proof}

The following reformulation of Theorem~\ref{t:equi'} will be very useful in applications. Call a Borel set
$E \subset \cUG_{m,n}$ a \emph{continuity set} if $\nu(\partial E)=0$, and denote by
$C^b(\cUG_{m,n})$ the space of bounded, continuous functions $\cUG_{m,n} \to \RR$.
Let
\[
    N_{m,n}(B)=\#\{\Lambda \in \Gr(m,n)(\QQ): H(\Lambda)\le B\}  
\]
be the number of rational points on $\Gr(m,n)$ of height at most $B$.
\begin{theorem}\label{t:equi''}
    Let 
    \[
        \nu_B = \frac{1}{N_{m,n}(B)} \sum_{H(\Lambda) \le B} \delta_{u(T_\Lambda)}
    \]
    be the sequence of probability measures on 
    $\cUG_{m,n}$ counting tangent lattices of points of bounded height.
    Then $\nu_B$ converges weakly to $\nu$ as $B\to\infty$, in the sense that
    \begin{enumerate}[label=(\roman*)]
        \item\label{enum:equi-int} $\lim_{B\to\infty} \int f \diff \nu_B = \int f \diff \nu \quad$
              for all  $f \in C^b(\cUG_{m,n})$,
        \item\label{enum:equi-cont-sets} $\lim_{B\to\infty }\nu_B(E) = \nu(E)\quad$
              for all continuity sets $E \subset \cUG_{m,n}$, and
        \item\label{enum:equi-closed-sets} $\limsup_{B\to \infty} \nu_B(E) \le \nu(E)\quad$  
              for all closed $E \subset \cUG_{m,n}$.
    \end{enumerate}
    Moreover, 
    \begin{enumerate}[label=(\roman*), resume]
    \item\label{enum:equi-counting} let $f\colon \cUG_{m,n}\to \RR_{\ge 0}$ be a bounded, continuous function.
    Let
    \begin{equation}\label{eq:modified-count}
        N_{m,n}(f;B) = \#\{\Lambda \in \Gr(m,n)(\QQ) : H(\Lambda) \le f(u(T_\Lambda)) B\}.
    \end{equation}
    Then
    \[
        \frac{N_{m,n}(f;B)}{N_{m,n}(B)}
        \to \int f \diff \nu,
        \qquad B\to \infty .
    \]
    \end{enumerate}
    
\end{theorem}

\begin{proof}
    The statements~\ref{enum:equi-int},~\ref{enum:equi-cont-sets},
    and~\ref{enum:equi-closed-sets} are equivalent.  It suffices to
    check~\ref{enum:equi-cont-sets} on a basis of the topology closed under
    finite intersections. Since the intersection of two boundary-controlled open
    sets is boundary-controlled and open, and since all open balls along local
    charts are boundary controlled, the assertion follows.  

    The last statement can be proved analogously to~\cite[Prop.~3.3~(b)]{peyre-duke} for continuous functions with compact support.
    It follows for arbitrary bounded, continuous functions $f$ using Prokhorov's theorem, cutting off $f$ outside progressively large compact subsets of $\cUG_{m,n}$. 
\end{proof}

\section{Consequences of equidistribution}\label{ssec:equi-and-freeness}

\subsection*{Equidistribution and freeness}

As a first application of the equidistribution theorem, we reprove Theorem~\ref{t:free}.
\begin{proof}[Alternative proof of Theorem~\ref{t:free}]
    Recall that $E_\ve(B) = \# \Omega_\ve(B)$, where
    \[
        \Omega_\ve(B) = \{\Lambda \in \Gr(m,n)(\QQ) : \ell(\Lambda) \le \ve , 
                          H(\Lambda) \le B\}.
    \]  
    Let
    \[
        \kappa = \limsup_{B\to\infty}
        \frac{
            E_\ve(B)
        }{
            N_{m,n}(B)
        }.
    \]
    We want to prove that $\kappa=0$.
    Note that $\mu(\alpha L) = \mu(L) - \log\alpha$ for all lattices $L$.
    Applying this to quotient lattices yields $\mu_{\min}(u(L)) = \mu_{\min}(L) - \mu(L)$.
    Moreover, recall that the logarithmic anticanonical height verifies $h(\Lambda) = \log H(\Lambda) = -\log \covol(T_\Lambda)$, whence 
    \begin{equation}\label{eq:slope-height}
        \mu(T_\Lambda) = \frac{h(\Lambda)}{m(n-m)}.
    \end{equation}
    Thus  $\ell(\Lambda)\le \ve$ is equivalent to
    \[
        \mu_{\min}(u(T_\Lambda)) \le \frac{\ve-1}{m(n-m)} h(\Lambda).
    \]
    For $R>0$, set
    \begin{align*}
        \Omega^{(0)}(R)
        & = \{ \Lambda \in \Gr(m,n)(\QQ) : H(\Lambda) \le R\} \quad \text{and} \\
        \Omega^{(1)}_{\ve}(R,B)
        & = \left\{ \Lambda \in \Gr(m,n)(\QQ) :
                \mu_{\min}(u(T_\Lambda)) \le \frac{\ve-1}{m(n-m)} \log R,\ 
                H(\Lambda) \le B \right\}.
    \end{align*}
    We note that
    \[
        \Omega_\ve(B) \subset \Omega^{(0)}(R) \cup \Omega^{(1)}_{\ve}(R,B),
    \]
    for all $R>0$.
    Clearly, 
    \[
        \limsup_{B\to\infty} \frac{\# \Omega^{(0)}(R)}{N_{m,n}(B)} = 0.
    \]
    To bound the number of elements of the second set, let 
    \[
        Z_\ve(R) = \left\{L \in \cUG_{m,n} : 
                     \mu_{\min}(L) \le \frac{\ve-1}{m(n-m)} \log R \right\},  
    \]
    and note that
    \[
        \Omega^{(1)}_{\ve}(R,B) = \{\Lambda \in \Gr(m,n)(\QQ) : u(T_\Lambda) \in Z_\ve(R), 
                                    \ H(\Lambda) \le B\}.
    \]
    Since $\mu_{\min}$ is clearly continuous,
    $Z_\ve(R)$ is closed. It follows that
    \[
        \kappa \le \limsup_{B\to\infty} \frac{
            \# \Omega^{(1)}_{\ve}(R,B)
         }{
             N_{m,n}(B)
         }
         \le \nu(Z_\ve(R)),
    \]
    by Theorem~\ref{t:equi''}~\ref{enum:equi-closed-sets}.
    The sets $Z_\ve(R)$ form a decreasing sequence whose intersection is empty, hence
    $\nu(Z_\ve(R)) \to 0$ as $R\to\infty$, completing the proof.
\end{proof}

\subsection*{Counting by maximal slope}
In~\cite[\S~9]{peyre-freedom}, Peyre puts forward an alternative way of involving
the well-shapedness of tangent lattices in point counting. He suggests ordering rational points $x$
on a Fano variety $V$ with a smooth integral model $V_\ZZ$ by the maximal slope of
their tangent lattice, instead of by height, so as to count the quantity
\[
    N_V^{\mu_{\max}}(B) = \# \{ x \in V(\QQ) : \mu_{\max}(\cT_{V_\ZZ,x}) \le \log B \}.    
\]
Since the logarithmic anticanonical height $h(x)=\log H(x)$ verifies
\[
    h(x) = \dim(V) \mu(\cT_{V_\ZZ,x}) \le \dim(V) \mu_{\max}(\cT_{V_\ZZ,x}),
\]
bounding the maximal slope automatically also bounds the anticanonical height,
yielding a trivial upper bound
\[
    N_V^{\mu_{\max}}(B) \le N_V(B^{\dim V}),
\]
where $N_V(B) = \{x\in V(\QQ) : H(x)\le B\}$. For $\PP^{n-1}=\Gr(1,n)$, this results in the upper bound
\[
    N_{\PP^{n-1}}^{\mu_{\max}}(B) \le c_{1,n} B^{n-1}(1+o(1)),
\]
by~\eqref{eq:schmidt}. In~\cite[Rem.~7.8]{peyre-freedom}, Peyre provides the lower bound
\[
    N_{\PP^{n-1}}^{\mu_{\max}}(B) \gg_{\eta} B^{n-1-\eta}
\]
for any $\eta>0$, and expects that $N_{\PP^{n-1}}^{\mu_{\max}}(B) \sim c_{1,n}^\prime B^{n-1}$
for a suitable constant $c_{1,n}^\prime >0$. 
Theorem~\ref{t:max-slope-intro}
confirms this expectation for projective space and provides an
analogous asymptotic formula for all other Grassmannians. 

\begin{proof}[Proof of Theorem~\ref{t:max-slope-intro}]
    Let
    \[
        N^{\mu_{\max}}_{m,n}(B) = \# \{\Lambda \in \Gr(m,n) : \mu_{\max} (T_\Lambda) \le \log B\}.
    \]
    Recall that $\mu(\alpha L) = \mu(L) -\log\alpha$. Applying this
    to all sublattices of $T_\Lambda$, we see that
    $\mu_{\max}(T_\Lambda) = \mu_{\max}(u(T_\Lambda)) + \mu(T_\Lambda)$.
    Using this and~\eqref{eq:slope-height}, the condition
    $e^{\mu_{\max}(T_\Lambda)} \le B$ can be seen to be equivalent to
    $$H(\Lambda)^{1/(m(n-m))} \le e^{-\mu_{\max}(u(T_\Lambda))} B.$$
    Hence
    \[
        N^{\mu_{\max}}_{m,n}\left(B^{ \frac{1}{m(n-m)} }\right) =
        N_{m,n}(f;B),
    \]
    for $f = e^{-m(n-m) \mu_{\max}}$, in the notation of~\eqref{eq:modified-count}.
    Since $\mu_{\max}$ is obviously continuous, so is $f$, and since
    $\mu_{\max}(L)\ge 0$ for unimodular $L$, the function $0\le f\le 1$ is bounded.
    An application of Theorem~\ref{t:equi''}~\ref{enum:equi-counting} now yields
    \[
        N^{\mu_{\max}}_{m,n}(B) \sim c_{m,n}^\prime B^{m(n-m)},
    \]
    where
    \begin{equation}\label{eq:c'_mn}
        c_{m,n}^\prime
        =
        c_{m,n}
        \int_{L \in \cUG_{m,n}}
        e^{-m(n-m)\mu_{\max}(L)} \diff \nu.
    \end{equation}

    It remains to prove that $0 < c_{m,n}^\prime < c_{m,n}$.
    Note that the non-empty open subsets $U_1=\{\mu_{\max}(L) < 1\}$ and
    $U_2 = \{\mu_{\max}(L) > 1\}$ of $\cUG_{m,n}$ have positive measure.
    Now $f|_{U_1} > e^{-m(n-m)}$, hence $\int f \diff \nu > 0$. This implies that $c_{m,n}'>0$.
    From $f\le 1$, we can deduce the inequality $\int f \diff \nu \le 1$,
    and since $f|_{U_2} \le e^{-m(n-m)} < 1$, the inequality is strict, whence
    $c_{m,n}'<c_{m,n}$.
\end{proof}

\end{document}